\documentclass[reqno]{amsart}
\usepackage{amsmath, amssymb, amsfonts}

\newtheorem{theorem}{Theorem}
\newtheorem{corollary}{Corollary}
\newtheorem{lemma}{Lemma}
\newtheorem{proposition}{Proposition}
\theoremstyle{remark}
\newtheorem{remark}{Remark}
\theoremstyle{definition}
\newtheorem{define}{Definition}
\newtheorem{example}{Example}

\newtheorem*{Acknowlegement}{Acknowlegement}

\newcommand{\paren}[1]{\left(#1\right)}
\newcommand{\bra}[1]{\left\{#1\right\}}

\begin{document}
\date{}

\title[On the automorphism groups of models]{On the automorphism groups of models in $\mathbb C^2$}

\author{Ninh Van Thu* and Mai Anh Duc**} 
\thanks{The research of the first author was supported in part by an NRF grant
2011-0030044 (SRC-GAIA) of the Ministry of Education, The Republic of Korea. The research of the second author was supported in part by an NAFOSTED grant
of Vietnam.}

\address{*Center for Geometry and its Applications,
Pohang University of Science and Technology,  Pohang 790-784, The Republic of Korea - and - Department of Mathematics, Vietnam National University at Hanoi, 334 Nguyen Trai, Thanh Xuan, Hanoi, Vietnam}
\email{thunv@vnu.edu.vn, thunv@postech.ac.kr}

\address{**Department of Mathematics, Hanoi
National University of Education, 136 Xuan Thuy str., Hanoi, Vietnam}
\email{ducphuongma@gmail.com}

\subjclass[2000]{Primary 32M05; Secondary 32H02, 32H50, 32T25.}
\keywords{Automorphism group, model, finite type point.}

\maketitle      
\begin{abstract}
In this note, we consider models in $\mathbb C^2$. The purpose of this note is twofold. We first  show a characterization of models in $\mathbb C^2$ by their noncompact automorphism groups. Then we give an explicit description for automorphism groups of models in $\mathbb C^2$.
\end{abstract}

\section{Introduction}  
For a domain $\Omega$ in the complex Euclidean space $\mathbb C^n$, the set of biholomorphic self-maps forms a group under the binary operation of composition of mappings, which is called \emph{automorphism group} ($\mathrm{Aut}(\Omega))$. The topology on $\mathrm{Aut}(\Omega)$ is that of uniform convergence on compact sets (i.e., the compact-open topology).

A boundary point $p\in \partial \Omega$ is called a \emph{boundary orbit accumulation point} if there exist a sequence $\{f_j\}\subset \mathrm{Aut}(\Omega)$ and a point $q\in \Omega$ such that $f_j(q)\to p$ as $j\to \infty$. The classification of domains with noncompact automorphism groups is pertinent to the study of the geometry of the boundary at an orbit accumulation point.  

In this note, we consider a model 
$$
M_H=\{(z_1,z_2)\in\mathbb C^2:\mathrm{Re}~z_2+H(z_1)<0\},
$$
where $H$ is a homogeneous subharmonic polynomial of degree $2m~ (m\geq 1)$ which contains no harmonic terms. It is a well-known result of F. Berteloot \cite{Ber2} that if $\Omega\subset\mathbb C^2$ is pseudoconvex, of D'Angelo finite type near a boundary orbit accumulation point, then $\Omega$ is biholomorphically equivalent to a model $M_H$. For the case $\Omega$ is strongly pseudoconvex, this result was proved by B. Wong \cite{W} and J. P. Rosay \cite{R}; indeed, the model is biholomorphically equivalent to the unit ball.  These results motivate the following several concepts.

A domain $\Omega\subset \mathbb C^2$ is said to satisfy \emph{Condition $(M_H)$ at $p\in \partial \Omega$} if there exist neighborhoods $U$ and $V$ of $p$ and $(0,0)$, respectively; a biholomorphism $\Phi$ from $U\cap \Omega$ onto $V\cap M_H$, which  extends homeomophically to $U\cap\partial \Omega$ such that $\Phi(p)=(0,0)$. In this circumstance, we say that a sequence $\{\eta_n\}\subset U\cap \Omega$ \emph{ converges tangentially to order $s~ (s>0)$} to $p$ if $\mathrm{dist}(\Phi(\eta_n),\partial M_H)\approx |\Phi(\eta_n)_1|^s$, where $\mathrm{dist}(z,\partial M_H)$ is the Euclidean distance from $z$ to $\partial M_H$ and $\Phi(\eta_n)_1$ is the first coordinate of $\Phi(\eta_n)$. Here and in what follows, $\lesssim$ and $\gtrsim$ denote inequality up to a positive constant. Moreover, we will use $\approx $ for the combination of $\lesssim$ and $\gtrsim$.

We first prove the following theorem.
\begin{theorem}\label{T1}
Let $\Omega$ be a domain in $\mathbb C^2$ and let $p\in \partial\Omega$. Suppose that $\Omega$ satisfies Condition $(M_H)$ at $p$ and there exist a sequence $\{f_n\}\subset \mathrm{Aut}(\Omega)$ and $q\in \Omega$ such that $\{f_n(q)\}$ converges tangentially to order $\leq 2m~ (=\deg(H))$ to $p$. Then
$\Omega$ is biholomorphically equivalent to the model $M_{H}$.
\end{theorem}

\begin{remark} Because of Condition $(M_H)$ at $p$, $\Omega$ is of finite type at $p$. Therefore, it is proved in \cite{Ber2} that $\Omega$ is biholomorphically equivalent to some model $M_{\tilde H}$, where $\tilde H$ is a subharmonic homogeneous polynomial. But we do not know the relationship between $H$ and $\tilde H$. Theorem \ref{T1} tells us that $\tilde H$ is exactly equal to $H$. 
\end{remark}
For a domain $\Omega$ in $\mathbb C^n$, the automorphism group is not easy to describe explicitly; besides, it is unknown in most cases. For instance, the automorphism groups of various domains are given in \cite{Chen, GKK, JP, Kim, Krantz, Shimizu, Sunada}.  Recently, explicit forms of automorphism groups of certain domains have been obtained in \cite{ ABP, BC1, BC2}. 

The second part of this note is to describe automorphism groups of models in $\mathbb C^2$. If a model is symmetric, i.e. $H(z_1)=|z_1|^{2m}$, then it is biholomorphically equivalent to the Thullen domain $E_{1,m}=\{(z_1,z_2)\in \mathbb C^2: |z_2|^2+|z_1|^{2m}<1\}$; the $\mathrm{Aut}(E_{1,m})$ is exactly the set of all biholomorphisms
$$
(z_1,z_2)\mapsto \Big(e^{i\theta_1}\frac{z_2-a}{1-\bar a z_2}, e^{i\theta_2}\frac {(1-|a|^2)^{1/2m}}{(1-\bar a z_2)^{1/m}}z_1\Big)
$$
for some $a\in \mathbb C$ with $|a|<1$ and $\theta_1,\theta_2\in \mathbb R$ (cf. \cite[Example 9, p.20]{GKK}). Let us denote by $\Omega_m=\{(z_1,z_2)\in\mathbb C^2:\mathrm{Re}~z_2+(\mathrm{Re}~z_1)^{2m}<0\}$. All the other models, which are not biholomorphically equivalent to $E_{1,m}$ or $\Omega_m$, will be treated together, as the generic case. Let us denote $T^1_t, T^2_t, R_\theta,S_\lambda$ by the following automorphisms:
\begin{equation*}
\begin{split}
T^1_t&: ~(z_1,z_2)\mapsto (z_1+it,z_2);\\
T^2_t&: ~(z_1,z_2)\mapsto (z_1,z_2+it);\\
R_{\theta}&: ~ (z_1,z_2)\mapsto (e^{i\theta}z_1,z_2);\\
S_\lambda&:~ (z_1,z_2)\mapsto (\lambda z_1,\lambda^{2m} z_2),
\end{split}
\end{equation*}
where $t\in\mathbb R$, $\lambda>0$, and $\exp (i\theta)$ is an $L$-root of unity ( see Section 4).   

With these notations, we obtain the following our second main result.
\begin{theorem}\label{T2} If $m\geq 2$, then 
\begin{itemize}
\item[(i)] $\mathrm{Aut}(\Omega_m)$ is generated by 
$$
\{T^1_t,T^2_t, R_{\pi},S_\lambda ~|~ t\in \mathbb R,\lambda>0 \};
$$
\item[(ii)] For any generic model $M_H$, $\mathrm{Aut}(M_H)$ is generated by 
$$
\{T^2_t, R_\theta,S_\lambda~|~ t\in \mathbb R,\lambda>0,\; \text{and}\; \exp (i\theta)\; \text{is an}\; L\text{-root of unity}  \}.
$$ 
\end{itemize}
\end{theorem}    

Let $S(\Omega)$ denote the set of all boundary accumulation points for $\mathrm{Aut}(\Omega)$. Then it follows from Theorem \ref{T2} that
\begin{itemize}
\item[(i)] $S( E_{1,m})=\{(e^{i\theta},0)\in \mathbb C^2\colon \theta\in [0,2\pi)\}$;
\item[(ii)] $S(\Omega_{m})=\{(it,is)\in \mathbb C^2\colon t,s\in \mathbb R\}\cup \{\infty\}$;
\item[(iii)] $S( M_H)=\{(it,0)\in \mathbb C^2\colon t\in \mathbb R\}\cup \{\infty\}$ for any generic model $M_H$.
\end{itemize}

We remark that, for any model $M_H$ in $\mathbb C^2$, $S(M_H)$ is a smooth submanifold of $\partial M_H$. Moreover, the D'Angelo type is constant and maximal along $S(M_H)$. In addition, the behaviour  of orbits in any model $M_H\subset \mathbb C^2$ is well-known. For instance, if there exist a point $q\in M_H$ and a sequence $\{f_n\}\subset \mathrm{Aut}(M_H)$ such that $\{f_n(q)\}$ converges to some boundary accumulation point $p\in S(M_H)\setminus \{\infty\}$, then it must converge tangentially to order $\leq \deg(H)$ to $p$. In the past twenty years, much attention has been given to the behaviour of orbits near an orbit accumulation point. We refer the reader to the articles \cite{F-W, IK, I-K}, and references therein for the development of related subjects.   

A typical consequence of Theorem \ref{T2} and the Berteloot's result~\cite{Ber2} is as follows.
\begin{corollary}\label{corollary2013}
Let $\Omega$ be a domain in $\mathbb C^2$. Suppose that there exist a point $q\in \Omega$ and a sequence $\{f_j\}\subset\mathrm{Aut}(\Omega)$ such that $\{f_j(q)\}$ converges to $p_\infty\in \partial \Omega$. Assume that the boundary of $\Omega$ is smooth, pseudoconvex, and of D'Angelo finite type near $p_\infty$ $(\tau(\partial \Omega,p_\infty)=2m)$. Then exactly one of the following alternatives holds:
\begin{itemize}
\item[(i)] If $\dim \mathrm{Aut}(\Omega)=2$ then 
$$
\Omega\simeq M_H=\{(z_1,z_2)\in \mathbb C^2:\mathrm{Re}~z_2+H(z_1)<0\},
$$
where $M_H$ is a generic model in $\mathbb C^2$ and $\deg(H)=2m$.
\item[(ii)]  If $\dim \mathrm{Aut}(\Omega)=3$ then 
$$
\Omega\simeq \Omega_m=\{(z_1,z_2)\in \mathbb C^2:\mathrm{Re}~z_2+(\mathrm{Re}~z_1)^{2m}<0\}.
$$
\item[(iii)]  If $\dim \mathrm{Aut}(\Omega)=4$ then 
$$
\Omega\simeq E_{1,m}=\{(z_1,z_2)\in \mathbb C^2:\mathrm{Re}~z_2+|z_1|^{2m}<0\}\simeq \{(z_1,z_2)\in \mathbb C^2:|z_2|^2+|z_1|^{2m}<1\}.
$$
\item[(iv)] If $\dim \mathrm{Aut}(\Omega)=8$ then 
$$
\Omega\simeq \mathbb B^2= \{(z_1,z_2)\in \mathbb C^2:|z_1|^2+|z_2|^2<1\}.
$$
\end{itemize} 
The dimensions $0,1,5,6,7$ cannot occur with $\Omega$ as above.
\end{corollary}

For the case that $\partial \Omega$ is real analytic and of D'Angelo finite type near a boundary orbit accumulation point (without the hypothesis of pseudoconvexity), a similar result as the above corollary was obtained in \cite{Verma} by using a different method. In addition, it was shown in \cite{B-P3} that a smoothly bounded $\Omega$ in $\mathbb C^2$ with real analytic boundary and  with noncompact automorphism group, must be biholomorphically equivalent to $E_{1,m}$.   

This paper is organized as follows. In Section 2, we review some basic notions needed later. In Section 3, we prove Theorem \ref{T1}. Finally, the proof of Theorem \ref{T2} is given in Section 4.

\section{Definitions and results}

First of all, we recall the following definitions.

\begin{define}[see \cite{D}] Let $\Omega\subset \mathbb C^n$ be a domain with $\mathcal{C}^\infty$-smooth boundary and $p\in \partial \Omega$. Then the D'Angelo type $\tau(\partial \Omega,p)$ of $\partial \Omega$ at $p$ is defined as
\begin{equation*}
 \tau(\partial \Omega,p):=\sup_{\gamma} \frac{\nu(\rho\circ \gamma)}{\nu(\gamma)},
\end{equation*}
where $\rho$ is a definining function of $\Omega$ near $p$, the supremum is taken over all germs of nonconstant holomorphic curves 
$\gamma: (\mathbb C,0)\to (\mathbb C^n,p)$. We say that $p$ is a point of \emph{finite type} if $\tau(\partial \Omega,p)<\infty$ and of \emph{infinite type} if otherwise.
\end{define}
\begin{define} Let $X$, $Y$ be complex spaces and $\mathcal F $$ \subset Hol(X,Y)$.
\begin{itemize}
\item[(i)] \  A sequence $\big\{f_j\big\} \subset \mathcal F$ is \emph{compactly divergent}
if for every compact set $K \subset X$ and for every compact set 
$L \subset Y$ there is a number $j_0 = j_0(K,L)$ such that 
$f_j(K) \cap  L = \emptyset$ for all $j \geq j_0$.
\item[(ii)] \  The family $\mathcal F$ is said to be not compactly divergent if
$\mathcal F$ contains  no compactly divergent subsequences.
\end{itemize}
\end{define}
\begin{define}\  A complex space $X$ is called \emph{taut} if for any family $\mathcal F \subset Hol(\Delta,X)$, there exists a subsequence $\{f_j\}\subset \mathcal F $ which is either convergent or compactly divergent, where $\Delta=\{z\in \mathbb C\colon |z|<1\}$.
\end{define}

We recall  the concept of Carath\'{e}odory kernel convergence of domains which is relevant 
to the discussion of scaling methods (see \cite{Duren}). Note that the local 
Hausdorff convergence can replace the normal convergence in case the domains in 
consideration are convex. 

\begin{define} [Carath\'{e}odory Kernel Convergence] Let $\{\Omega_\nu\}$ be a sequence 
of domains in $\mathbb C^{n}$ such that $p\in\bigcap\limits_{\nu=1}^{\infty}\Omega_\nu$. 
If $p$ is an interior point of $\bigcap\limits_{\nu=1}^{\infty}\Omega_\nu$, the 
Carath\'{e}odory kernel $\hat \Omega$ at $p$ of the sequence $\{\Omega_\nu\}$, is defined 
to be the largest domain containing $p$ having the property that each compact subset of 
$\hat \Omega$ lies in all but a finite number of the domains $\Omega_\nu$. If   $p$ is not an 
interior point of $\bigcap\limits_{\nu=1}^{\infty}\Omega_\nu$, the Carath\'{e}odory kernel 
$\hat \Omega$ is $\{p\}$. The sequence $\{\Omega_\nu\}$  is said to converge to its kernel at 
$p$ if every subsequence of $\{\Omega_\nu\}$ has the same kernel at $p$.

We shall say that a sequence $\{\Omega_\nu\}$ of domains in $\mathbb C^n$ 
converges normally to $\hat \Omega$ (denoted by $\lim \Omega_\nu=\hat \Omega$) 
if there exists a point $p\in\bigcap\limits_{\nu=1}^{\infty}\Omega_\nu$ such that 
$\{\Omega_\nu\}$ converges to its Carath\'{e}odory kernel $\hat \Omega$ at $p$.      
\end{define}
Now we recall  several results which will be used later on. The following proposition is a generalization of the theorem of Greene-Krantz \cite{GK} (cf.~\cite{Thai}).
\begin{proposition} \label{P1}  Let $\{A_j\}_{j=1}^\infty$ and
 $\{\Omega_j\}_{j=1}^\infty$ be sequences of domains in a complex manifold $M$ 
with $\lim A_j=A_0$ and $\lim \Omega_j=\Omega_0$ for some $($uniquely determined $)$ 
domains $A_0$, $\Omega_0$ in $M$. Suppose that $\{f_j: A_j \to \Omega_j\} $ is a 
sequence of biholomorphic maps. Suppose also that the sequence $\{f_j: A_j\to M \}$ 
converges uniformly on compact subsets of $ A_0$ to a holomorphic map $F:A_0\to M $ 
and the sequence $\{g_j:=f^{-1}_j: \Omega_j\to M \}$ converges uniformly on compact 
subsets of $\Omega_0$ to a holomorphic map $G:\Omega_0\to M $.
Then one of the following two assertions holds. 
\begin{enumerate}
\item[(i)] The sequence $\{f_j\}$ is compactly divergent, i.e., for each compact 
set $K\subset A_0$ and each compact set $L\subset \Omega_0$, 
there exists an integer $j_0$ such that $f_j(K)\cap L=\emptyset$ for $j\geq j_0$, or
\item[(ii)] There exists a subsequence $\{f_{j_k}\}\subset \{f_j\}$  
such that the sequence $\{f_{j_k}\}$ converges uniformly on compact subsets 
of $A_0$ to a biholomorphic map $F: A_0 \to \Omega_0$.
\end{enumerate}
\end{proposition}

In closing this section we recall the following lemma (see \cite{Ber2}).
\begin{lemma}[F. Berteloot]\label{L:4}
Let $\sigma_\infty$ be a subharmonic function of class $\mathcal{C}^2$ on $\mathbb C$ such that $\sigma_\infty(0)=0 $ and $\int_{\mathbb C} \bar\partial \partial\sigma_\infty =+\infty$. Let $\{\sigma_k\}$ be a sequence of subharmonic functions on $\mathbb C$ which  converges uniformly on compact subsets of $\mathbb C$ to $\sigma_\infty$. Let $\Omega$ be any domain in a complex manifold of dimension $m\ (m\geq 1)$ and let  $z_0$ be a fixed point in $\Omega$. Denote by $M_k$ the domain in $ \mathbb C^2$ defined by 
$$M_k=\{(z_1,z_2)\in \mathbb C^2: \mathrm{Re}~z_2+\sigma_k(z_1)<0\}.$$ 
Then any sequence $h_k\in Hol(\Omega, M_k)$ such that $\{h_k(z_0), k\geq 1\}\Subset M_\infty $ admits a subsequence which converges uniformly on compact subsets of $\Omega$ to an element of $Hol(\Omega, M_\infty)$. 
\end{lemma}
 
\section{Asymptotic behaviour of orbits in a model in $\mathbb C^2$}
Let $P$ be a subharmonic polynomial. Let us denote by $M_P$ the model given by
$$
M_P=\{(z_1,z_2)\in\mathbb C^2: \rho(z_1,z_2):=\mathrm{Re}~z_2+P(z_1)<0\}.
$$
Let $\Omega$ be a domain in $\mathbb C^2$. Suppose that $\partial \Omega$ is pseudoconvex, finite type, and smooth of class $\mathcal{C}^\infty$ near a boundary point $p\in \partial \Omega$. In \cite{Ber2}, F. Berteloot proved that if $p$ is a boundary orbit accumulation point for $\mathrm{Aut}(\Omega)$, then $\Omega$ is biholomorphically equivalent to a model $M_H$, where $H$ is a homogeneous subharmonic polynomial of degree $2m$ which contains no harmonic terms with $\|H\|=1$. Here and in what follows, denote by $\|P\|$ the maximum of absolute values of the coefficients of a polynomial $P$. Let us denote by $\mathcal{P}_{2m}$ the space of real valued polynomials on $\mathbb C$ with degree less than or equals to $2m$ and which do not contain any harmonic term and by 
$$
\mathcal{H}_{2m} =\{H\in \mathcal{P}_{2m}\;\text{such that}\; deg(H)=2m \;\text{and}\; H\; \text{is homogeneous and subharmonic}\}.
$$

From now on, let $H\in \mathcal{H}_{2m}$ be as in Theorem \ref{T1}. Taking the risk of confusion we employ the notation
$$
H_{j} := \frac{\partial^j H}{\partial z_1^j};\; H_{j,\bar q}:=\frac{\partial^{j+q} H}{\partial z_1^j\partial \bar z_1^q}
$$
throughout the paper for all $j,q\in \mathbb N^*$.

For each $a=(a_1,a_2)\in \mathbb C^2$, let us define 
$$
H_a(w_1)=\frac{1}{\epsilon(a)} \sum_{j,q>0} \frac{H_{j,\bar q}(a_1) }{(j+q)!} \tau(a)^{j+q} w_1^j\bar w_1^q,
$$
where $\epsilon(a)=|\mathrm{Re}~a_2+H(a_1)| $ and $\tau(a)$ is chosen so that $\|H_a\|=1$. We note that $\sqrt{\epsilon(a)}\lesssim\tau(a)\lesssim\epsilon(a)^{1/(2m)}$. Denote by $\phi_a$ the holomorphic map
\begin{equation*}
\begin{split}
\phi_a: \mathbb C^2&\to \mathbb C^2\\ 
z &\mapsto w=\phi_a(z), 
\end{split}
\end{equation*}
given by  
\begin{equation*}
\begin{cases}
{w}_2=\dfrac{1}{\epsilon(a)}\Big[z_2-a_{2}-\epsilon(a)+2\sum\limits_{j=1}^{2m}\frac{H_j(a_{1})}{j!}(z_1-a_{1})^j   \Big]\\
{w}_1 =\dfrac{1}{\tau(a)}[z_1-a_{1}].
\end{cases}
\end{equation*}
It is easy to check that $\phi_a$ maps biholomorphically $M_H$ onto $M_{H_a}$ and $\phi_a(a)=(0,-1)$. 

Now let us consider a domain $\Omega$ in $\mathbb C^2$ satisfying Condition $(M_H)$ at a boundary point $p\in \partial \Omega$. With no loss of generality, we can assume $p=(0,0)$ and
$$
\Omega\cap U=\{(z_1,z_2)\in U: \rho(z_1,z_2)=\mathrm{Re}~z_2+H(z_1)<0\}.
$$
Assume that there exist a sequence $\{f_n\}\subset \mathrm{Aut}(\Omega)$ and a point $q\in M_H$ such that $\eta_n:=f_n(q)\to (0,0)$ as $n\to\infty$.
\begin{remark}\label{remark}
By Proposition 2.1 in \cite{Ber2}, $\Omega$ is taut and after taking a subsequence we 
may assume that for each compact subset $K\subset \Delta$ there exists a positive integer $n_0$ such that $f_n(K)\subset \Omega\cap U $ 
for every $n\geq n_0$.
\end{remark}

 Since $\|H_{\eta_n}\|=1$, passing to a subsequence if necessary, we can assume that $\lim H_{\eta_n}= H_\infty$, where $H_\infty\in \mathcal{P}_{2m}$ and $\|H_\infty\|=1$. 
\begin{proposition}\label{prop1} $\Omega$ is biholomorphically equivalent to $M_{H_\infty}$.
\end{proposition}
\begin{proof}
Let $\psi_n :=\phi_{\eta_n}\circ f_n$ for each $n\in \mathbb N^*$ and consider the following sequence of biholomorphisms
\begin{equation*}
\begin{split}
\psi_n: f_n^{-1}(\Omega\cap U)&\to M_{H_{\eta_n}}\\ 
q &\mapsto (0,-1).
\end{split}
\end{equation*}
By Lemma \ref{L:4} and by Remark \ref{remark}, after taking  a subsequence we may assume that $\{\psi_n\}$ converges uniformly on any compact subsets of $\Omega$ to a holomorphic map $g: \Omega \to M_{H_\infty}$. In the other hand, since $\Omega$ is taut we can assume that $\{\psi^{-1}_n\}$ converges also uniformly on any compact subset of $M_{H_\infty}$ to a holomorphic map $\tilde g: M_{H_\infty}\to M_H $. Therefore it follows from Proposition \ref{P1} that $g$ is biholomorphic, and hence $\Omega$ is biholomorphically equivalent to $M_{H_\infty}$.
\end{proof}
\begin{remark} $\mathrm{dist}(\eta_n, \partial M_H)\approx \epsilon_n:=|\rho(\eta_n)|$.
\end{remark}
\begin{remark}

\noindent 
i) Let $\{\eta_n\}$ be a sequence in $M_H$  which converges 
tangentially to order $2m$ to $(0,0)$. Set $\epsilon_n:=|\rho(\eta_n)|\approx |\eta_{n1}|^{2m}$. 
Then we have
\begin{equation*}
\begin{split}
|\mathrm{Re}~\eta_{n2}|&=|\epsilon_n+H(\eta_{n1})|\\
                 &\lesssim |\eta_{n1}|^{2m}.
\end{split}
\end{equation*}

\noindent
ii) Suppose that $\{\eta_n\}$ is a sequence in $M_H$  which converges 
tangentially to order $<2m$ to $(0,0)$. Then we have $|\eta_{n1}|^{2m}=o(\epsilon_n)$ and  we thus obtain the following estimate
\begin{equation*}
\begin{split}
|\mathrm{Re}~\eta_{n2}|&=|\epsilon_n+H(\eta_{n1})|\\
                 &\approx |\epsilon_n|.
\end{split}
\end{equation*}

\end{remark}
\begin{lemma}\label{lem7} If  $\{\eta_n\}\subset M_H$ converges tangentially to order $2m$ to $(0,0)$, then $deg(H_\infty)=2m$ and moreover $M_{H_\infty}$ is biholomorphically equivalent to $M_H$.
\end{lemma}
\begin{proof}
Since $\{\eta_n\}$ converges tangentially to order $2m$ to $(0,0)$, it follows that
$|\eta_{n1}|^{2m} \approx \epsilon_n \approx d(\eta_n, \partial \Omega)$. Let $a_{j,\bar q}(\eta_n):=\frac{H_{j,\bar q}(\eta_{n1})\tau(\eta_n)^{j+q}}{(j+q)!\epsilon_n}$ for each $j,q>0$ with $j+q\leq 2m$. Then we have the following estimate
\begin{equation*}
\begin{split}
|a_{j,\bar q}(\eta_n) |\lesssim \frac{|\eta_{n1}|^{2m-j-q}\tau(\eta_n)^{j+q}}{(j+q)!\epsilon_n} \lesssim \Big(\frac{\tau(\eta_n)}{|\eta_{n1}|}\Big)^{j+q}.
\end{split}
\end{equation*}
Since $\|H_{\eta_n}\|=1$, we have $\tau(\eta_n)\gtrsim |\eta_{n1}|\approx \epsilon_n^{1/(2m)}$, and therefore $\tau({\eta}_n)\approx \epsilon_n^{1/(2m)}$. This implies that $deg(H_\infty)=2m$. 
Without loss of generality we can assume that $\lim \frac{\eta_{n1}}{\epsilon_n^{1/(2m)}}=\alpha$ and $\lim \frac{\tau(\eta_{n})}{\epsilon_n^{1/(2m)}}=\beta$.
 We note that
\begin{equation*}
\begin{split}
a_{j,\bar q}(\eta_n)&=\frac{H_{j,\bar q}(\eta_{n1})\tau(\eta_n)^{j+q}}{(j+q)!\epsilon_n}\\
&=\frac{1}{(j+q)!}\Big(\frac{\tau(\eta_n)}{\epsilon_n^{1/(2m)}}\Big)^{j+q} H_{j,\bar q}\Big(\frac{\eta_{n1}}{\epsilon_n^{1/(2m)}}\Big) 
\end{split}
\end{equation*}
for any $j,q>0$. Then we obtain $\lim a_{j,\bar q}(\eta_n)=\frac{1}{(j+q)!}\beta^{j+q} H_{j,\bar q}(\alpha) w_1^j \bar w_1^q$  for each $j,q >0$; hence 
\begin{equation*}
\begin{split}
H_\infty(w_1)&=\sum_{j,q>0}\frac{1}{(j+q)!} \beta^{j+q} H_{j,\bar q}(\alpha)w_1^j\bar w_1^q\\
&=H(\alpha+\beta w_1)-H(\alpha)-2\text{Re} \sum_{j=1}^{2m} \frac{H_j(\alpha )}{j!}(\beta w_1)^j. 
\end{split}
\end{equation*}
So, the holomorphic map given by 
\begin{equation*}
\begin{cases}
{t}_2=w_2-H(\alpha)-2\sum\limits_{j=1}^{2m}\frac{H_j(\alpha)}{j!} (\beta w_1)^j  \\
{t}_1 =\alpha+\beta w_1
\end{cases}
\end{equation*}
is biholomorphic from $M_{H_\infty}$ onto $M_H$.
\end{proof}
\begin{lemma}\label{lem8} If  $\{\eta_n\}\subset M_H$ converges tangentially to order $< 2m$ to $(0,0)$, then $H_\infty=H$.
\end{lemma}
\begin{proof}
It is easy to see that $\tau({\eta}_n) \lesssim \epsilon_n^{1/(2m)}$. On the other hand, 
since $|\eta_{n1}|^{2m}=o( |\epsilon_n|)$, we have for $j,q\in\mathbb N$ with $j,q>0, j+q<2m$ that
\begin{equation*}
\begin{split}
|a_{j,\bar q}({\eta}_n)|&\lesssim \frac{|\eta_{n1}|^{2m-j-q}\epsilon^{(j+q)/(2m)}_n}{(j+q)!\epsilon_n}\\
                                      &\lesssim  \Big(\frac{|\eta_{n1}|^{2m}}{\epsilon_n}\Big)^{\frac{2m-j-q}{2m}}.
\end{split}
\end{equation*}
Therefore $\lim a_{j,\bar q}({\eta}_n)=0$ for any $j,q>0$ with $j+q<2m$,  and thus $H_\infty=H$.  Hence, the proof is complete.
\end{proof}

\begin{proof}[Proof of Theorem \ref{T1}]
Let $\Omega$ and $\{f_n\}$ be a domain and a sequence, respectively, as in Theorem \ref{T1}. Then, after a change of coordinates, we can assume that $p=(0,0)$ and
$$
\Omega\cap U=\{(z_1,z_2)\in U: \rho(z_1,z_2)=\mathrm{Re}~z_2+H(z_1)<0\}.
$$
Moreover, we may also assume that $\eta_n:=f_n(q)\in U\cap M_H$ for all $n\in \mathbb N^*$. Therefore, it follows from Proposition \ref{prop1}, Lemma \ref{lem7}, and Lemma \ref{lem8} that $\Omega$ is biholomorphically equivalent to $M_H$, which finishes the proof.
\end{proof}
In the case that $\{\eta_n\}$ converges tangentially to order $>2m$ to $(0,0)$, we obtain the following proposition. 

\begin{proposition} Let $\{\eta_n\}\subset M_H$ be a sequence which converges tangentially to order $ >2 m$ to $(0,0)$.
If there exist $j,q>0$ with $j+q<2m$ such that 
$$\Big| \frac{\partial^{j+q}H}{\partial z_1^j\partial \bar z_1^q}(\eta_{n1})\Big|
\approx |\eta_{n1}|^{2m-j-q},$$
then $\tau({\eta}_n)=o(\epsilon_n^{1/(2m)})$, and thus $deg(H_\infty)<2m$.
\end{proposition}
\begin{proof} 
Suppose otherwise that $\tau({\eta}_n)\approx \epsilon_n^{1/(2m)}$. Then
since $\epsilon_n=o(|{\eta}_{n1}|^{2m})$, one gets
\begin{equation*}
\begin{split}
|a_{j,\bar q}({\eta}_n)| &\approx \frac{|\eta_{n1}|^{2m-j-q}\epsilon^{(j+q)/(2m)}_n}{(j+q)!\epsilon_n}\\
                                      &\approx  \Big(\frac{|\eta_{n1}|^{2m}}{\epsilon_n}\Big)^{\frac{2m-j-q}{2m}}.
\end{split}
\end{equation*}
This implies that
$$\lim_{n\to \infty} a_{j,\bar q}({\eta}_n)|=+\infty,$$
which is a contradiction. Thus, the proof is complete.
\end{proof}
\begin{example}
Let $E_{1,2}:=\{(z_1,z_2)\in \mathbb C^2:\mathrm{Re}~z_2+|z_1|^4<0\}$. Then the sequence $\{( 1/\sqrt[4]{n},-2/n)\}$
converges tangentially to order $4$ to $(0,0)$. But the sequence $\{(1/\sqrt[4]{n},-1/n-1/n^2)\} $
converges tangentially to order $8$ to $(0,0)$. 

 Let $\rho(z_1,z_2)=\mathrm{Re}~z_2+|z_1|^4$ and let $\eta_n=( 1/\sqrt[4]{n},-1/n-1/n^2)$ for every $n\in \mathbb N^*$. 
We see that $\rho(\eta_n)=-1/n-1/n^2+1/n=-1/n^2\approx -\mathrm{dist}(\eta_n,\partial E_{1,2})$. Set
$\epsilon_n=|\rho(\eta_n)|=1/n^2$. Then
\begin{equation*}
\begin{split}
&\rho(z_1,z_2)=
\mathrm{Re}(z_2)+ |\frac{1}{\sqrt[4]{n}}+z_1-\frac{1}{\sqrt[4]{n}}|^4\\
                      &= \mathrm{Re}(z_2)+\frac{1}{n}+ |z_1-\frac{1}{\sqrt[4]{n}}|^4+\frac{1}{\sqrt{n}}(2\text{Re}(z_1-\frac{1}{\sqrt[4]{n}}))^2
+ \frac{4}{\sqrt[4]{n}} |z_1-\frac{1}{\sqrt[4]{n}}|^2 \mathrm{Re}(z_1-\frac{1}{\sqrt[4]{n}})\\
&+\frac{4}{\sqrt{n}}\frac{1}{\sqrt[4]{n}} \mathrm{Re}(z_1-\frac{1}{\sqrt[4]{n}}) +\frac{2}{\sqrt{n}}|z_1-\frac{1}{\sqrt[4]{n}}|^2\\
&= \mathrm{Re}(z_2)+\frac{1}{n} +\frac{4}{\sqrt{n}\sqrt[4]{n}}\mathrm{Re}(z_1-\frac{1}{\sqrt[4]{n}})
+\frac{2}{\sqrt{n}}\mathrm{Re}((z_1-\frac{1}{\sqrt[4]{n}})^2) + |z_1-\frac{1}{\sqrt[4]{n}}|^4\\
&+\frac{4}{\sqrt{n}}|z_1-\frac{1}{\sqrt[4]{n}}|^2+\frac{4}{\sqrt[4]{n}}|z_1-\frac{1}{\sqrt[4]{n}}|^2 \mathrm{Re}(z_1-\frac{1}{\sqrt[4]{n}}).
\end{split}
\end{equation*}

A direct calculation shows that $\tau_n:=\tau(\eta_n)=\frac{1}{2 n^{3/4}}$ for all $n=1,2,\ldots$ and thus the automorphism $\phi_{{\eta}_n}$ is given by
$$\phi_{{\eta}_n} ^{-1}(w_1,w_2)= \Big (\frac{1}{\sqrt[4]{n}}+\tau_n w_1,\epsilon_n w_2-\frac{1}{n}- \frac{4}{\sqrt{n}\sqrt[4]{n}}\tau_n w_1-\frac{2}{\sqrt{n}}\tau_n^2 w_1^2 \Big);$$
\begin{equation*}
\begin{split}
\epsilon_n^{-1}\rho\circ \phi_{{\eta}_n} ^{-1}(w_1,w_2)&= \epsilon_n^{-1}\rho \Big (\frac{1}{\sqrt[4]{n}}+\tau_n w_1,\epsilon_n w_2-\frac{1}{n}- \frac{4}{\sqrt{n}\sqrt[4]{n}}\tau_n w_1-\frac{2}{\sqrt{n}}\tau_n^2 w_1^2 \Big)\\
&=\mathrm{Re}(w_2) + \frac{1}{16n}|w_1|^4+|w_1|^2
+ \frac{1}{2\sqrt[4]{n}}|w_1|^2\mathrm{Re}(w_1).
\end{split}
\end{equation*} 
We now show that there do not exist a sequence $\{f_n\}\subset \mathrm{Aut}(E_{1,2})$ and $a\in E_{1,2}$ such that
$\eta_n=f_n(a)\to (0,0)\in \partial E_{1,2}$ as $n\to \infty$. Indeed, suppose that there exist such a sequence $\{f_n\}$ 
and such a point $a\in E_{1,2}$. Then by Proposition \ref{prop1}, $E_{1,2}$ is biholomorphically equivalent to the following 
domain $D:=\{(w_1,w_2)\in \mathbb C^2: \mathrm{Re}~w_2+ |w_1|^2<0\}\simeq \mathbb B^2 $. It is impossible.
\end{example}
\section{Automorphism group of a model in $\mathbb C^2$}
In this section, we consider a model
$$
M_H:=\{(z_1,z_2)\in \mathbb C^2:\mathrm{Re}~z_2+H(z_1)<0\},
$$
 where 
\begin{equation}\label{eq1613}
H(z_1)=\sum_{j=1}^{2m-1}a_{2m-j}z_1^j\bar z_1^{2m-j}=a_m|z_1|^{2m}+2 \sum_{j=1}^{m-1} |z_1|^{2j}\mathrm{Re}(a_j z_1^{2m-2j})
\end{equation}
is a nonzero real valued homogeneous polynomial of degree $2m$, with $a_j\in\mathbb C$ and $a_j=\overline{a_{2m-j}}$. We will give the explicit description of $\mathrm{Aut}(M_H)$.

The D'Angelo type of $\partial M_H$ is given by the following.
\begin{lemma}\label{L6}
$\tau(\partial M_H, (\alpha,-H(\alpha)+it))=m_\alpha$ for all $\alpha\in \mathbb C$ and for all $t\in \mathbb R$, where
$$
m_\alpha=\min\{j+q~|~j,q>0, \frac{\partial^{j+q} H(\alpha)}{\partial z_1^j \partial \bar z_1^q}\ne 0\}.
$$
\end{lemma}
\begin{proof}
By the following change of variables
\begin{equation*}
\begin{cases}
{w}_2=z_2+H(\alpha)+2\sum\limits_{j=1}^{2m}\frac{H_j(\alpha)}{j!}(z_1-\alpha)^j   \\
{w}_1 =z_1-\alpha,
\end{cases}
\end{equation*}
the defining function for $M_H$ is now given by
 $$
\rho(w_1,w_2)=\mathrm{Re}~w_2+ \sum_{j,q>0} \frac{H_{j,\bar q}(\alpha) }{(j+q)!}  w_1^j\bar w_1^q.
$$
By a computation, we get $\tau(\partial M_H, (\alpha,-H(\alpha)+it))=m_\alpha$, and thus the proof is complete.
\end{proof}

  Let $P_k(\partial M_H)$ the set of all points in $\partial M_H$ of D'Angelo type $k$ ($k$ is either a positive integer or infinity). Let us denote by
$\Gamma:=\{(z_1,-H(z_1)+it)~|~ t\in\mathbb R, z_1\in\mathbb C \;\text{with} \;\mathrm{Re}(e^{i\nu} z_1)=0\}$ if $H(z_1)=a\big[(2\mathrm{Re}(e^{i\nu}z_1))^{2m}-2\mathrm{Re}(e^{i\nu}z_1)^{2m}\big]$ for some $a\in\mathbb R^*$ and for some $\nu\in [0,2\pi)$ and by $\Gamma:=\{(0,it)~|~t\in\mathbb R\}$ if otherwise.

\begin{lemma}\label{L14} If $m\geq 2$, then
$P_{2m}(\partial M_H)=\Gamma$ and $\tau(\partial M_H, p)<2m$ for all $p\in \partial M_H\setminus \Gamma$.
\end{lemma}
\begin{proof}

It is not hard to show that $\Gamma\subset P_{2m}(\partial M_H)$. Now let $p=(\alpha, -H(\alpha)+it)~(\alpha\ne 0)$  be any boundary point in $\partial M_H\setminus \Gamma$. By Lemma \ref{L6}, we see that $\tau(\partial M_H, p)=m_\alpha\leq 2m$. We will prove that $\tau(\partial M_H, p)<2m$. Indeed, suppose that, on the contrary, $\tau(\partial M_H, p)=m_\alpha=2m$. This implies that $H_{j,\bar q}(\alpha)=0$ for all $j,q>0$ and $j+q<2m$ and thus $H_{1,\bar 1}(\alpha+z_1)=H_{1,\bar 1}(z_1)$ for all $z_1\in \mathbb C$. Let $f(x,y):=H_{1,\bar 1}(x+iy)$ for all $z_1=x+iy\in \mathbb C$. By a change of affine coordinates in $\mathbb C$, we may assume that $\alpha=(1,0)$ and thus $f(x+1,y)=f(x,y)$ for all $(x,y)\in \mathbb R^2$. Hence, for each $y\in\mathbb R$ $f(x,y)$ is a periodic polynomial in $x$, and thus     
$f(x,y)$ does not depend on $x$, i.e., $f(x,y)=\beta y^{2m-2}$ for some $\beta\in\mathbb R$. 

Therefore by the above, we conclude that $H_{1,\bar 1}(z_1)=\beta(\mathrm{Re}(e^{i\nu} z_1))^{2m-2}$ for some $\beta\in\mathbb R^*$ and for some $\nu\in [0,2\pi)$ and $\alpha$ satisfies $\mathrm{Re}(e^{i\nu}\alpha)=0$. Since $H$ is a homogeneous polynomial of degree $2m$ without harmonic terms, it is easy to show that $H(z_1)=a\big[(2\mathrm{Re}(e^{i\nu} z_1))^{2m}-2\mathrm{Re}(e^{i\nu} z_1)^{2m}\big]$ for some $a\in\mathbb R^*$ and $(\alpha, -H(\alpha)+ it)\in \Gamma$, which is impossible. Thus the proof is complete.
\end{proof}
We recall the following lemma, proved by F. Berteloot (see  \cite{Ber2}), which is the main ingredient in the proof of Theorem \ref{T2}. 
\begin{lemma}[F. Berteloot]\label{L7} Let $Q\in \mathcal{P}_{2m}$ and $H\in \mathcal{H}_{2m}$. Suppose that $\psi: M_H\to M_Q$ is a biholomorphism. Then there exist $t_0\in\mathbb R$ and $z_0\in \partial M_Q$ such that $\psi$ and $\psi^{-1}$ extend to be holomorphic in neighborhoods of $(0,it_0)$ and $z_0$, respectively. Moreover, the homogeneous part of higher degree in $Q$ is equal to $\lambda H(e^{i\nu}z)$ for some $\lambda>0$ and $\nu \in [0,2\pi)$.
\end{lemma}
\begin{proof}
According to \cite{B-P2}, there exists a holomorphic function $\phi$ on $M_Q$ which is continuous on $\overline{M_Q}$ such that $|\phi|<1$ for $z\in M_Q$ and tends to $1$ at infinity. Let $\psi: M_H\to M_Q$ be a biholomorphism. We claim that there exists $t_0\in \mathbb R$ such that $\lim_{x\to 0^-} \inf |\psi (0',x+it_0)|<+\infty$. 
Indeed, if this would not be the case, the function $\phi\circ\psi$ would be equal to $1$ on the half plane $\{(z_1,z_2)\in \mathbb C^2\colon\mathrm{Re}~z_2<0, z_1=0\}$ and this is impossible since $|\phi|<1$ for $|z|\gg 1$. Therefore, we may assume that there exists a sequence $x_k<0$ such that $\lim x_k=0$ and $\lim\psi (0,x_k+it_0)=z_0\in \partial M_Q$. It is proved in \cite{Ber} that under these circumstances $\psi$ extends homeomorphically to $\partial M_H$ on some neighbourhood of $(0,it_0)$. Then the result of Bell (see \cite{Bell}) shows that this extension is actually diffeomorphic. Moreover, it follows from \cite[Theorem 3]{Bell-Catlin} (see also \cite{DP, SV}) $\psi$ and $\psi^{-1}$ extend to be holomorphic in neighborhoods of $(0,it_0)$ and $z_0$, respectively. Therefore, the conclusion follows easily. 
\end{proof}

Now we recall two basic integer valued invariants used in the normal form construction in \cite{Kol}. Let $l=m_0<m_1<\cdots<m_p\leq m$ be indices in (\ref{eq1613}) for which $a_{m_i}\ne 0$. Denote by $L$ the greatest common divisor of $2m-2m_0, 2m-2m_1, \ldots, 2m-2m_p$. If $l=m$, then $H(z_1)=a_m |z_1|^{2m}~(a_m>0)$ and it is known that $M_H$ is biholomorphically equivalent to the domain
$$
E_{1,m}=\{(z_1,z_2)\in \mathbb C^2\colon |z_2|^2+|z_1|^{2m}<1\}.
$$
The automorphism group of $E_{1,m}$ is well-known (see \cite[Example 9, p. 20]{GKK}). So, in what follows we only consider the case $l<m$. Moreover, we consider the model $\Omega_m=\{(z_1,z_2)\in\mathbb C:\mathrm{Re}~z_2+(\mathrm{Re}~z_1)^{2m}<0\}$ and others which are not biholomorphically equivalent to it. 
\begin{remark} If $H(z_1)=a\big[(2\mathrm{Re}(e^{i\nu}z_1))^{2m}-2\mathrm{Re}(e^{i\nu}z_1)^{2m}\big]$ for some $a>0$ and for some $\nu\in [0,2\pi)$, then $M_H\simeq  \Omega_m$ and $L=2$. Indeed, $L=2$ is obvious. Now let us denote by $\Phi:\mathbb C^2\to \mathbb C^2$ the bihomorphism defined by $w_2=z_2-2a(e^{i\nu}z_1)^{2m};w_1=2a^{1/2m}e^{i\nu}z_1$. Then it is easy to check that $\Omega_m=\Phi(M_H)$. Hence, the assertion follows.
\end{remark}

\begin{lemma}\label{L15}
$H(\exp (i\theta)z_1)=H(z_1)$ for all $z_1 \in \mathbb C$ if and only if $\exp(i\theta)$ is an $L$-root of unity.
\end{lemma}
\begin{proof}
We have
\begin{equation*}
H(\exp (i\theta)z_1)=a_m|z_1|^{2m}+2 \sum_{j=0}^{p}\paren{|z_1|^{2j}\mathrm{Re}\bra{a_{m_j}\exp (i(2m-2m_j)\theta)z_1^{2m-2m_j}}}
\end{equation*}
for all $z_1\in \mathbb C$. Hence, we conclude that $H(\exp (i\theta)z_1)=H(z_1)$ for all $z_1 \in \mathbb C$  if and only if $\exp (i(2m-2m_j)\theta)=1$ for every $j=0,\ldots,p$, which proves the assertion.
\end{proof}
\begin{proof}[Proof of Theorem \ref{T2}]
For $t\in\mathbb R$, $\lambda>0$, and any $L$-root of unity $\exp (i\theta)$, consider the mappings
\begin{equation*}
\begin{split}
T^1_t&: ~(z_1,z_2)\mapsto (z_1+it,z_2);\\
T^2_t&: ~(z_1,z_2)\mapsto (z_1,z_2+it);\\
R_{\theta}&: ~ (z_1,z_2)\mapsto (e^{i\theta}z_1,z_2);\\
S_\lambda&:~ (z_1,z_2)\mapsto (\lambda z_1,\lambda^{2m} z_2).
\end{split}
\end{equation*}
It is easy to check that $T^2_t, R_{\theta}, S_\lambda$ are in $\mathrm{Aut}(M_H)$ and moreover $T^1_t\in \mathrm{Aut}(M_H)$ if $H(z_1)=(\mathrm{Re}~z_1)^{2m}$ for all $z_1\in \mathbb C$. Now let $f=(f_1,f_2)$ be any biholomorphism of $M_H$. It follows from Lemma \ref{L7} that there exist $p\in \Gamma$ and $q\in \Gamma$ such that $f$ and $f^{-1}$ extend to be holomorphic in neighborhoods of $p$ and $q$, respectively, and $f(p)=q$. Replacing $f$ by its composition with reasonable translations $T^2_t, T^1_{t}$, we may assume that $p=q=(0,0)$, and there exist neighborhoods $U_1$ and $U_2$ of $(0,0)$ such that $U_2\cap \partial M_H= f(U_1\cap \partial M_H)$, and $f$ and $f^{-1}$ are holomorphic in $U_1$ and $U_2$, respectively. Moreover, $f$ is a local CR diffeomorphism between $U_1\cap \partial M_H$ and $U_2\cap \partial M_H$.

Let us denote by $\mathcal{H}=\{z\in \mathbb C: \mathrm{Re}~z<0\}$. We now define $g_1(z_2):=f_1(0,z_2)$ and $g_2(z_2):=f_2(0,z_2)$ for all $z_2\in \mathcal{H}$. It follows from Lemma \ref{L6} that $f(U_1\cap \Gamma)=U_2\cap \Gamma$. Consequently, $g_1(it)=0$ for all $-\epsilon_0<t<\epsilon_0$ with $\epsilon_0>0$ small enough. By the Schwarz Reflection Principle and the Identity Theorem, we have $g_1(z_2)= 0$ for all $z_2\in \mathcal{H}$. This also implies that $\mathrm{Re}~ f_2(0,z_2)<0$, and thus $g_2\in \mathrm{Aut}(\mathcal{H})$. Since $g_2(0)=0$, it is known that $g_2(z_2)=\dfrac{\alpha z_2}{1+i\beta z_2}$ for some $\alpha\in \mathbb R^*$ and $\beta\in\mathbb R$.  

Now we are going to prove that $f$ is biholomorphic between neighborhoods of the origin. To do this, it suffices to show that $J_f(0,0)\ne 0$~( a simillar proof shows that $J_{f^{-1}}(0,0)\ne 0$).  To derive a contradiction, we suppose that $J_f(0,0)=0$.  By the above we can write
$$
f(z_1,z_2)=\big(z_1 a(z_1,z_2), g_2(z_2)+z_1 b(z_1,z_2)\big),
$$
where $a$ and $b$ are holomorphic functions defined on neighborhoods of $(0,0)$, respectively. By shrinking $U_1$ if necessary, we can assume that $a,b$ are defined on $U_1$.

Take derivative of $f$ at points $(0,z_2)$ we have
\begin{equation*}
df(z_1,z_2)=
 \left( \begin{array}{ccc}
a (z_1,z_2)& z_1 a_{z_1}(z_1,z_2)  \\
b(z_1,z_2)& {g_2}'(z_2)+ z_1 b_{z_2}(z_1,z_2)
\end{array} \right).
\end{equation*}
Therefore we obtain $J_f(0,z_2)=a(0,z_2){g_2}'(z_2) $ for every $z_2$ small enough. We note that $J_f(0,z_2)\ne 0 $ for all $z_2\in \mathcal{H}$, ${g_2}'(0)=\alpha\ne 0$ , and $J_f(0,0)=0$. This implies that $a(z_1,z_2)=O(|z|)$. 

Since $f(z_1,z_2)\in \overline{M_H}\cap U_2$ for all $(z_1,z_2)\in \overline{M_H}\cap U_1$,
$$
\mathrm{Re} \big(g_2(z_2)+z_1 b(z_1,z_2)\big)+H\big(z_1 a(z_1,z_2)\big)\leq 0
$$
for all $(z_1,z_2)\in \overline{M_H}\cap U_1$. Because of the invariance of $\overline{M_H}$ under any map $S_t\; (t>0)$, one gets
\begin{equation}\label{eq2013-1}
\mathrm{Re} \big(g_2(t^{2m}z_2)+tz_1 b(tz_1,t^{2m}z_2)\big)+H\big(tz_1 a(tz_1,t^{2m}z_2)\big)\leq 0
\end{equation}
for every $(z_1,z_2)\in \overline{M_H}\cap U_1$ and for every $t\in (0,1)$.

Expand the function $b$ into the Taylor series at the origin so that
$$
b(z_1,z_2)=\sum_{j,k=0}^\infty b_{j,k} z_1^j z_2^k,
$$
where $b_{j,k}\in \mathbb C$ for all $j,k\in \mathbb N$. Hence the equation (\ref{eq2013-1}) can be re-written as
\begin{equation}\label{eq2013-2}
\begin{split}
\rho\circ f(tz_1, t^{2m}z_2)&=
\mathrm{Re} \Big(\alpha\frac{t^{2m}z_2}{1+i\beta t^{2m}z_2}+tz_1\sum_{j,k=0}^\infty b_{j,k}(tz_1)^j(t^{2m}z_2)^k\Big)\\
&\quad +H(tz_1 a(tz_1,t^{2m}z_2))\leq 0
\end{split}
\end{equation}
for every $(z_1,z_2)\in \overline{M_H}\cap U_1$ and for every $t\in (0,1)$.

Now let us denote by $j_0=\min\{j~| ~b_{j,0}\ne 0\}$ if $b(z_1,0)\not \equiv 0$ and $j_0=+\infty$ if otherwise. We divide the argument into three cases as follows.

\smallskip

{\bf Case 1.} {\boldmath $0\leq j_0\leq 2m-2$.}  Note that we can choose $\delta_0>0$ and $\epsilon_0>0$ such that $H(z_1)<\epsilon_0$ for all $|z_1|<\delta_0$. Since $(-\epsilon_0, z_1)\in U\cap M_H$ for all $|z_1|<\delta_0$, taking $\lim_{t\to 0^+}\frac{1}{t^{j_0+1}}\rho\circ f(tz_1, t^{2m}\epsilon_0)$ we obtain
$\mathrm{Re}(b_{j_0,0} z_1^{j_0+1})\leq 0$ for all $|z_1|<\delta_0$, which leads to a contradiction. 

\smallskip

{\bf Case 2.} {\boldmath $ j_0= 2m-1$.} It follows from (\ref{eq2013-2}) that
\begin{equation*}
\begin{split}
\lim_{t\to 0^+}\frac{1}{t^{2m}}\rho\circ f(tz_1, t^{2m}z_2)&=\mathrm{Re}(\alpha z_2+b_{2m-1,0}z_1^{2m})=0
\end{split}
\end{equation*}
 for all $(z_1,z_2)\in U_1$ with $\mathrm{Re}~z_2+H(z_1)=0$. This implies that
$H(z_1)=\mathrm{Re} (\frac{b_{2m-1,0}}{\alpha}z_1^{2m})$ for all $|z_1|<\delta_0$ with $\delta_0>0$ small enough. It is absurd since $H$ contains no harmonic terms.

\smallskip

{\bf Case 3.} {\boldmath $ j_0>2m-1$.}  Fix a point $(z_1,z_2)\in U_1\cap \partial M_H $ with $\mathrm{Re}(z_2)\ne 0$. From (\ref{eq2013-2}) one has
\begin{equation*}
\begin{split}
\lim_{t\to 0^+}\frac{1}{t^{2m}}\rho\circ f(tz_1, t^{2m}z_2)=\mathrm{Re}(\alpha z_2)=0,
\end{split}
\end{equation*}
which is impossible.

Altogether, we conclude that $f$ is a local biholomorphism between neighborhoods $U_1$ and $U_2$ of the origin satisfying $f(U_1\cap \partial M_H)=U_2\cap \partial M_H$. Therefore by \cite[Corollary 5.3, p. 909]{Kol} and the Identity Theorem, we have
$$
f(z_1,z_2)=( \lambda e^{i\theta} z_1, \lambda^{2m}z_2)
$$
for all $(z_1,z_2)\in M_H$, where $e^{i\theta}$ is an $L$-root of unity. Thus $f=S_\lambda\circ R_\theta$, and hence the proof is complete.
\end{proof}

\begin{Acknowlegement} We would like to thank Prof. Kang-Tae Kim, Prof. Do Duc Thai, and Dr. Hyeseon Kim for their precious discussions on this material.
\end{Acknowlegement}

\end{document}